\documentclass[12pt]{amsart}%
\usepackage{amsmath}
\usepackage{amsfonts}
\usepackage{amssymb}
\usepackage{graphicx}
\usepackage{comment}
\usepackage[english]{babel}
\usepackage{graphicx}
\usepackage[margin= 0.7in]{geometry}
\usepackage{setspace}
\providecommand{\U}[1]{\protect\rule{.1in}{.1in}}
\providecommand{\U}[1]{\protect\rule{.1in}{.1in}}
\newtheorem{theorem}{Theorem}

\newtheorem{lemma}[theorem]{Lemma}

\newtheorem{proposition}[theorem]{Proposition}
\newtheorem{remark}[theorem]{Remark}

\begin{document}
\title{Decompositions of Rational Gabor Representations}
\author{Vignon Oussa}
\address{Dept.\ of Mathematics \\
Bridgewater State University\\
Bridgewater, MA 02325 U.S.A.\\}
\date{January $2015$}
\maketitle

\begin{abstract}
Let $\Gamma=\left\langle T_{k},M_{l}:k\in\mathbb{Z}^{d},l\in B\mathbb{Z}^{d}\right\rangle $ be a group of unitary operators where $T_{k}$ is a
translation operator and $M_{l}$ is a modulation operator acting on
$L^{2}\left(  \mathbb{R}^{d}\right) .$ Assuming that $B$ is a non-singular
rational matrix of order $d$ with at least one entry which is not an integer,
we obtain a direct integral irreducible decomposition of the Gabor
representation which is defined by the isomorphism $\pi:\left(  \mathbb{Z}_{m}\times B\mathbb{Z}^{d}\right)  \rtimes\mathbb{Z}^{d}\rightarrow\Gamma$
where $\pi\left(  \theta,l,k\right)  =e^{\frac{2\pi i}{m}\theta}M_{l}T_{k}.$
We also show that the left regular representation of $\left(  \mathbb{Z}_{m}\times B\mathbb{Z}^{d}\right)  \rtimes\mathbb{Z}^{d}$ which is identified
with $\Gamma$ via $\pi$ is unitarily equivalent to a direct sum of
$\mathrm{card}\left(  \left[  \Gamma,\Gamma\right]  \right)  $ many disjoint
subrepresentations of the type: $L_{0},L_{1},\cdots,L_{\mathrm{card}\left(
\left[  \Gamma,\Gamma\right]  \right)  -1}\ $such that for $k\neq1$ the
subrepresentation $L_{k}$ of the left regular representation is disjoint from
the Gabor representation. Additionally, we compute the central decompositions of the representations $\pi$ and $L_1.$ These decompositions are then exploited to give a new proof of the Density Condition of Gabor systems (for the rational case). More precisely, we prove that $\pi$ is equivalent to a subrepresentation of $L_1$ if and only if $|\det B|\leq 1.$ We also derive characteristics of vectors $f$ in $L^{2}(\mathbb{R})^{d}$ such that $\pi(\Gamma)f$ is a Parseval frame in
$L^{2}(\mathbb{R})^{d}.$ 

\end{abstract}


\section{Introduction}

The concept of applying tools of abstract harmonic analysis to time-frequency
analysis, and wavelet theory is not a new idea \cite{Baggett, bag1, bag2,
hartmut, moussa}. For example in \cite{Baggett}, Larry Baggett gives a direct
integral decomposition of the Stone-von Neumann representation of the discrete
Heisenberg group acting in $L^{2}(\mathbb{R}).$ Using his decomposition, he
was able to provide specific conditions under which this representation is
cyclic. In Section $5.5,$ \cite{hartmut} the author obtains a characterization
of tight Weyl-Heisenberg frames in $L^{2}\left(\mathbb{R}\right)$ using
the Zak transform and a precise computation of the Plancherel measure of a
discrete type $I$ group. In \cite{moussa}, the authors present a thorough
study of the left regular representations of various subgroups of the reduced
Heisenberg groups. Using well-known results of admissibility of unitary
representations of locally compact groups, they were able to offer new
insights on Gabor theory.

Let $B$ be a non-singular matrix of order $d$ with real entries. For each
$k\in\mathbb{Z}^{d}$ and $l\in B\mathbb{Z}^{d},$ we define the corresponding
unitary operators $T_{k},M_{l}$ such that $T_{k}f\left(  t\right)  =f\left(
t-k\right)  \text{ and }M_{l}f\left(  t\right)  =e^{-2\pi i\left\langle
l,t\right\rangle }f\left(  t\right) $ for $f\in L^{2}\left(  \mathbb{R}^{d}\right)  .$ The operator $T_{k}$ is called a shift operator, and the
operator $M_{l}$ is called a modulation operator. Let $\Gamma$ be a subgroup
of the group of unitary operators acting on $L^{2}\left(  \mathbb{R}^{d}\right)  $ which is generated by the set $\left\{  T_{k},M_{l}:k\in\mathbb{Z}^{d},l\in B\mathbb{Z}^{d}\right\}  .$ We write $\Gamma
=\left\langle T_{k},M_{l}:k\in\mathbb{Z}^{d},l\in B\mathbb{Z}^{d}\right\rangle
.$ The commutator subgroup of $\Gamma$ given by $\left[  \Gamma,\Gamma\right]
=\left\{e^{2\pi i\left\langle l,Bk\right\rangle }:l,k\in\mathbb{Z}^{d}\right\}$ is a subgroup of the one-dimensional torus $\mathbb{T}$. Since
$\left[\Gamma,\Gamma\right]$ is always contained in the center of the group, then
$\Gamma$ is a nilpotent group which is generated by $2d$ elements. Moreover, $\Gamma$ is given the discrete topology, and as such it is a locally compact group. We observe that if $B$ has at least one irrational entry, then it is a non-abelian group with an infinite center. If $B$ only has rational entries with at least one entry which is not an integer, then $[\Gamma,\Gamma]$ is a finite group, and $\Gamma$ is a non-abelian group which is regarded as a finite extension of an abelian group. If all entries of $B$ are integers, then $\Gamma$ is abelian, and clearly $[\Gamma,\Gamma]$ is trivial. Finally, it is worth mentioning that $\Gamma$ is a type $I$ group if and only if $B$ only has rational entries \cite{Thoma}.

It is easily derived from the work in Section $4$, \cite{moussa} that if $B$
is an integral matrix, then the Gabor representation $\pi:B\mathbb{Z}%
^{d}\times\mathbb{Z}^{d}\rightarrow\Gamma\subset\mathcal{U} \left(
L^{2}\left( \mathbb{R}^{d}\right)  \right) $ defined by $\pi\left(
l,k\right)  =M_{l}T_{k}$ is equivalent to a subrepresentation of the left
regular representation of $\Gamma$ if and only if $B$ is a unimodular matrix.
The techniques used by the authors of \cite{moussa} rely on the decompositions
of the left regular representation and the Gabor representation into their
irreducible components. The group generated by the operators $M_{l}$ and
$T_{k}$ is a commutative group which is isomorphic to $\mathbb{Z}_{d}\times
B\mathbb{Z}_{d}.$ The unitary dual and the Plancherel measure for discrete
abelian groups are well-known and rather easy to write down. Thus, a precise
direct integral decomposition of the left regular representation is easily
obtained as well. Next, using the Zak transform, the authors decompose the
representation $\pi$ into a direct integral of its irreducible components.
They are then able to compare both representations. As a result, one can
derive from the work in the fourth section of \cite{moussa} that the
representation $\pi$ is equivalent to a subrepresentation of the left regular
representation if and only if $B$ is a unimodular matrix. The main objective
of this paper is to generalize these ideas to the more difficult case where
$B\in GL\left(  d,\mathbb{Q}\right)$ and $\Gamma$ is not a commutative group.

Let us assume that $B$ is an invertible rational matrix with at least one
entry which is not an element of $\mathbb{Z}.$ Denoting the inverse transpose
of a given matrix $M$ by $M^{\star},$ it is not too hard to see that there
exists a matrix $A\in GL\left(  d,\mathbb{Z}\right)  $ such that
$\Lambda=B^{\star}\mathbb{Z}^{d}\cap\mathbb{Z}^{d}=A\mathbb{Z}^{d}.$ Indeed, a
precise algorithm for the construction of $A$ is described on Page $809$ of
\cite{Unified}. Put
\begin{equation}
\Gamma_{0}=\left\langle \tau,M_{l},T_{k}:l\in B\mathbb{Z}^{d},k\in\Lambda
,\tau\in\left[  \Gamma,\Gamma\right]  \text{ }\right\rangle \label{Gamma0}%
\end{equation}
and define $\Gamma_{1}=\left\langle \tau,M_{l}:l\in B\mathbb{Z}^{d},\tau
\in\left[  \Gamma,\Gamma\right]  \text{ }\right\rangle .$ Then $\Gamma_{0}$ is
a normal abelian subgroup of $\Gamma.$ Moreover, we observe that $\Gamma_{1}$
is a subgroup of $\Gamma_{0}$ of infinite index. Let $m$ be the number of
elements in $\left[  \Gamma,\Gamma\right]  .$ Clearly, since $B$ has at least
one rational entry which is not an integer, it must be the case that $m>1$.
Furthermore, it is easy to see that there is an isomorphism $\pi:\left(
\mathbb{Z}_{m}\times B\mathbb{Z}^{d}\right)  \rtimes\mathbb{Z}^{d}%
\rightarrow\Gamma\subset\mathcal{U}\left(  L^{2}\left(  \mathbb{R}^{d}\right)
\right) $ defined by $\pi\left(  j,Bl,k\right)  =e^{\frac{2\pi ji}{m}}%
M_{Bl}T_{k}$. The multiplication law on the semi-direct product group $\left(
\mathbb{Z}_{m}\times B\mathbb{Z}^{d}\right)  \rtimes\mathbb{Z}^{d}$ is
described as follows. Given arbitrary elements
\[
\left(  j,Bl,k\right)  ,\left(  j_{1},Bl_{1},k_{1}\right)  \in\left(
\mathbb{Z}_{m}\times B\mathbb{Z}^{d}\right)  \rtimes\mathbb{Z}^{d},
\]
we define $\left(  j,Bl,k\right)  \left(  j_{1},Bl_{1},k_{1}\right)  =\left(
\left(  j+j_{1}+\omega\left(  l_{1},k\right)  \right)  \operatorname{mod}%
m,B\left(  l+l_{1}\right)  ,k+k_{1}\right) $ where $\omega\left(
l_{1},k\right)  \in\mathbb{Z}_{m}$, and $\left\langle Bl_{1},k\right\rangle
=\frac{\omega\left(  l_{1},k\right)  }{m}.\ $We call $\pi$ a rational Gabor
representation. It is also worth observing that $\pi^{-1}\left(  \Gamma
_{0}\right)  =\left(  \mathbb{Z}_{m}\times B\mathbb{Z}^{d}\right)  \times
A\mathbb{Z}^{d}$ and $\pi^{-1}\left(  \Gamma_{1}\right)  =\left(
\mathbb{Z}_{m}\times B\mathbb{Z}^{d}\right)  \times\left\{  0\right\}
\simeq\mathbb{Z}_{m}\times B\mathbb{Z}^{d}.$ Throughout this work, in order to
avoid cluster of notations, we will make no distinction between $\left(
\mathbb{Z}_{m}\times B\mathbb{Z}^{d}\right)  \rtimes\mathbb{Z}^{d}$ and
$\Gamma$ and their corresponding subgroups.


The main results of this paper are summarized in the following propositions.
Let
\[
\Gamma=\left\langle T_{k},M_{l}:k\in\mathbb{Z}^{d},l\in B\mathbb{Z}%
^{d}\right\rangle
\]
and assume that $B$ is an invertible rational matrix with at least one entry
which is not an integer. Let $L$ be the left regular representation of
$\Gamma.$

\begin{proposition}
\label{L}The left regular representation of $\Gamma$ is decomposed as
follows:
\begin{equation}
L\simeq\oplus_{k=0}^{m-1}\int_{\frac{\mathbb{R}^{d}}{B^{\star}\mathbb{Z}^{d}%
}\times\frac{\mathbb{R}^{d}}{A^{\star}\mathbb{Z}^{d}}}^{\oplus}\mathrm{Ind}%
_{\Gamma_{0}}^{\Gamma}\chi_{\left(  k,\sigma\right)  }\text{ }d\sigma.
\label{decomp}%
\end{equation}
Moreover, the measure $d\sigma$ in (\ref{decomp}) is a Lebesgue measure, and
(\ref{decomp}) is not an irreducible decomposition of $L.$
\end{proposition}

\begin{proposition}
\label{operator}The Gabor representation $\pi$ is decomposed as follows:
\begin{equation}
\pi\simeq\int_{\frac{\mathbb{R}^{d}}{\mathbb{Z}^{d}}\times\frac{\mathbb{R}%
^{d}}{A^{\star}\mathbb{Z}^{d}}}^{\oplus}\mathrm{Ind}_{\Gamma_{0}}^{\Gamma}%
\chi_{\left(  1,\sigma\right)  }\text{ }d\sigma. \label{comp}%
\end{equation}
Moreover, $d\sigma$ is a Lebesgue measure defined on the torus $\frac
{\mathbb{R}^{d}}{\mathbb{Z}^{d}}\times\frac{\mathbb{R}^{d}}{A^{\star
}\mathbb{Z}^{d}}$ and (\ref{comp}) is an irreducible decomposition of $\pi.$
\end{proposition}

\vskip0.5cm It is worth pointing out here that the decomposition of $\pi$
given in Proposition \ref{operator} is consistent with the decomposition
obtained in Lemma $5.39,$ \cite{hartmut} for the specific case where $d=1$ and
$B$ is the inverse of a natural number larger than one.

\begin{proposition}
\label{pi}Let $m$ be the number of elements in the commutator subgroup
$\left[  \Gamma,\Gamma\right]  .$ There exists a decomposition of the left
regular representation of $\Gamma$ such that
\begin{equation}
L\simeq{\bigoplus\limits_{k=0}^{m-1}}L_{k}%
\end{equation}
and for each $k\in\left\{  0,1,\cdots,m-1\right\}  ,$ the representation
$L_{k}$ is disjoint from $\pi$ whenever $k\neq1.$ Moreover, the Gabor
representation $\pi$ is equivalent to a subrepresentation of the
subrepresentation $L_{1}$ of $L$ if and only if $\left\vert \det B\right\vert
\leq1.$
\end{proposition}

Although this problem of decomposing the representations $\pi$ and $L$ into
their irreducible components is interesting in its own right, we shall also
address how these decompositions can be exploited to derive interesting and
relevant results in time-frequency analysis. The proof of Proposition \ref{pi}
allows us to state the following:

\begin{enumerate}
\item There exists a measurable set $\mathbf{E\subset\mathbb{R}}^{d}$ which is
a subset of a fundamental domain for the lattice $B^{\star}\mathbb{Z}%
^{d}\times A^{\star}\mathbb{Z}^{d}$, satisfying
\[
\mu\left(  \mathbf{E}\right)  =\frac{1}{\left\vert \det\left(  B\right)
\det\left(  A\right)  \right\vert \dim\left(  l^{2}\left(  \Gamma/\Gamma
_{0}\right)  \right) }
\]
where $\mu$ is the Lebesgue measure on $\mathbb{R}^{d}\times\mathbb{R}^{d}.$

\item There exists a unitary map
\begin{equation}
\mathfrak{A:}\int_{\mathbf{E}}^{\oplus}\left(  \oplus_{k=1}^{\ell\left(
\sigma\right)  }l^{2}\left(  \frac{\Gamma}{\Gamma_{0}}\right)  \right)
d\sigma\rightarrow L^{2}\left(  \mathbb{R}^{d}\right)
\end{equation}
which intertwines $\int_{\mathbf{E}}^{\oplus}\left(  \oplus_{k=1}^{\ell\left(
\sigma\right)  }\mathrm{Ind}_{\Gamma_{0}}^{\Gamma}\left(  \chi_{\left(
1,\sigma\right)  }\right)  \right)  d\sigma\text{ with }\pi$ such that, the
multiplicity function $\ell$ is bounded, the representations $\chi_{\left(
1,\sigma\right)  }$ are characters of the abelian subgroup $\Gamma_{0}$ and
\begin{equation}
\int_{\mathbf{E}}^{\oplus}\left(  \oplus_{k=1}^{\ell\left(  \sigma\right)
}\mathrm{Ind}_{\Gamma_{0}}^{\Gamma}\left(  \chi_{\left(  1,\sigma\right)
}\right)  \right)  d\sigma\label{central}%
\end{equation}
is the central decomposition of $\pi$ (Section $3.4.2,$ \cite{hartmut}).
\end{enumerate}

Moreover, for the case where $|\det(B)|$ $\leq1,$ the multiplicity function
$\ell$ is bounded above by the number of cosets in $\Gamma/\Gamma_{0}$ while
if $|\det(B)|>1,$ then the multiplicity function $\ell$ is bounded but is
greater than the number of cosets in $\Gamma/\Gamma_{0}$ on a set
$\mathbf{E^{\prime}}\subseteq\mathbf{E}$ of positive Lebesgue measure. This
observation that the upper-bound of the multiplicity function behaves
differently in each situation may mistakenly appear to be of limited
importance. However, at the heart of this observation, lies a new
justification of the well-known Density Condition of Gabor systems for the
rational case (Theorem $1.3,$ \cite{Han}). In fact, we shall offer a new proof
of a rational version of the Density Condition for Gabor systems in
Proposition \ref{ratmatrix}. \vskip0.5cm

\noindent It is also worth pointing out that the central decomposition of
$\pi$ as described above has several useful implications. Following the
discussion on Pages $74$-$75,$ \cite{hartmut}, the decomposition given in
(\ref{central}) may be used to:

\begin{enumerate}
\item Characterize the commuting algebra of the representation $\pi$ and its center.

\item Characterize representations which are either quasi-equivalent or
disjoint from $\pi$ (see \cite{hartmut} Theorem $3.17$ and Corollary $3.18$).
\end{enumerate}

Additionally, using the central decomposition of $\pi,$ in the case where the
absolute value of the determinant of $B$ is less or equal to one, we obtain a
complete characterization of vectors $f$ such that $\pi\left(  \Gamma\right)
f$ is a Parseval frame in $L^{2}\left(  \mathbb{R}^{d}\right)  .$

\begin{proposition}
\label{main2} Let us suppose that $\left\vert \det B\right\vert \leq1.$ Then
\begin{equation}
\pi\simeq\int_{\mathbf{E}}^{\oplus}\left(  \oplus_{k=1}^{\ell\left(
\sigma\right)  }\mathrm{Ind}_{\Gamma_{0}}^{\Gamma}\left(  \chi_{\left(
1,\sigma\right)  }\right)  \right)  d\sigma
\end{equation}
with $\ell\left(  \sigma\right)  \leq\mathrm{card}\left(  \Gamma/\Gamma
_{0}\right)  $. Moreover, $\pi\left(  \Gamma\right)  f$ is a Parseval frame in
$L^{2}\left(  \mathbb{R}^{d}\right)  $ if and only if $f=\mathfrak{A}\left(
a\left(  \sigma\right)  \right)  _{\sigma\in\mathbf{E}}$ such that for
$d\sigma$-almost every $\sigma\in\mathbf{E,}$ $\left\Vert a\left(
\sigma\right)  \left(  k\right)  \right\Vert _{l^{2}\left(  \frac{\Gamma
}{\Gamma_{0}}\right)  }^{2}=1\text{ for }1\leq k\leq\ell\left(  \sigma\right)
$ and for distinct $k,j\in\left\{  1,\cdots,\ell\left(  \sigma\right)
\right\}  $, $\left\langle a\left(  \sigma\right)  \left(  k\right)  ,a\left(
\sigma\right)  \left(  j\right)  \right\rangle =0.$
\end{proposition}

This paper is organized around the proofs of the propositions mentioned above. In
Section $2$, we fix notations and we revisit well-known concepts such as
induced representations and direct integrals which are of central importance.
The proof of Proposition \ref{L} is obtained in the third section. The proofs
of Propositions \ref{operator}, \ref{pi} and examples are given in the fourth
section. Finally, the last section contains the proof of Proposition
\ref{main2}.

\section{Preliminaries}

Let us start by fixing our notations and conventions. All representations in
this paper are assumed to be unitary representations. Given two equivalent
representations $\pi$ and $\rho,$ we write that $\pi\simeq\rho.$ We use the
same notation for isomorphic groups. That is, if $G$ and $H$ are isomorphic
groups, we write that $G\simeq H.$ All sets considered in this paper will be
assumed to be measurable. Given two disjoint sets $A$ and $B$, the disjoint
union of the sets is denoted $A\overset{\cdot}{\cup}B.$ Let $\mathbf{H}$ be a
Hilbert space. The identity operator acting on $\mathbf{H}$ is denoted
$1_{\mathbf{H}}.$ The unitary equivalence classes of irreducible unitary
representations of $G$ is called the unitary dual of $G$ and is denoted
$\widehat{G}.$

Several of the proofs presented in this work rely on basic properties of
induced representations and direct integrals. The following discussion is
mainly taken from Chapter $6,$ \cite{Folland}. Let $G$ be a locally compact
group, and let $K$ be a closed subgroup of $G.$ Let us define $q:G\rightarrow
G/K$ to be the canonical quotient map and let $\varphi$ be a unitary
representation of the group $K$ acting in some Hilbert space which we call
$\mathbf{H.}$ Next, let $\mathbf{K}_{1}$\textbf{ }be the set of continuous
$\mathbf{H}$-valued functions $f$ defined over $G$ satisfying the following properties:

\begin{enumerate}
\item $q\left(  \text{support }\left(  f\right)  \right)  $ is compact,

\item $f\left(  gk\right)  =\varphi\left(  k\right)  ^{-1}f\left(  g\right)  $
for $g\in G$ and $k\in K.$
\end{enumerate}

Clearly, $G$ acts on the set $\mathbf{K}_{1}$ by left translation. Now, to
simplify the presentation, let us suppose that $G/K$ admits a $G$-invariant
measure. We remark that in general, this is not always the case. However, the
assumption that $G/K$ admits a $G$-invariant measure holds for the class of
groups considered in this paper. We construct a unitary representation of $G$
by endowing $\mathbf{K}_{1}$ with the following inner product:
\[
\left\langle f,f^{\prime}\right\rangle =\int_{G/K}\left\langle f\left(
g\right)  ,f^{\prime}\left(  g\right)  \right\rangle _{\mathbf{H}}\text{
}d\left(  gK\right)  \text{ for }f,f^{\prime}\in\mathbf{K}_{1}.
\]
Now, let $\mathbf{K}$ be the Hilbert completion of the space $\mathbf{K}_{1}$
with respect to this inner product. The translation operators extend to
unitary operators on $\mathbf{K}$ inducing a unitary representation
$\mathrm{Ind}_{K}^{G}\left(  \varphi\right)  $ which is defined as follows:%
\[
\mathrm{Ind}_{K}^{G}\left(  \varphi\right)  \left(  x\right)  f\left(
g\right)  =f\left(  x^{-1}g\right)  \text{ for }f\in\mathbf{K.}%
\]
We notice that if $\varphi$ is a character, then the Hilbert space
$\mathbf{K}$ can be naturally identified with $L^{2}\left(  G/H\right)  .$
Induced representations are natural ways to construct unitary representations.
For example, it is easy to prove that if $e$ is the identity element of $G$
and if $1$ is the trivial representation of $\left\{  e\right\}  $ then the
representation $\mathrm{Ind}_{\left\{  e\right\}  }^{G}\left(  1\right)  $ is
equivalent to the left regular representation of $G.$ Other properties of
induction such as induction in stages will be very useful for us. The reader
who is not familiar with these notions is invited to refer to Chapter $6$ of
the book of Folland \cite{Folland} for a thorough presentation.

We will now present a short introduction to direct integrals; which are
heavily used in this paper. For a complete presentation, the reader is
referred to Section $7.4,$ \cite{Folland}. Let $\left\{  H_{\alpha}\right\}
_{\alpha\in A}$ be a family of separable Hilbert spaces indexed by a set $A.$
Let $\mu$ be a measure defined on $A.$ We define the direct integral of this
family of Hilbert spaces with respect to $\mu$ as the space which consists of
vectors $\varphi$ defined on the parameter space $A$ such that $\varphi\left(
\alpha\right)  $ is an element of $H_{\alpha}$ for each $\alpha\in A$ and
$\int_{A}\left\Vert \varphi\left(  \alpha\right)  \right\Vert _{H_{\alpha}%
}^{2}d\mu\left(  \alpha\right)  <\infty$ with some additional measurability
conditions which we will clarify. A family of separable Hilbert spaces
$\left\{  H_{\alpha}\right\}  _{\alpha\in A}$ indexed by a Borel set $A$ is
called a field of Hilbert spaces over $A.$ A map $\varphi:A\rightarrow
{\displaystyle\bigcup\limits_{\alpha\in A}}H_{\alpha}\text{ such that }%
\varphi\left(  \alpha\right)  \in H_{\alpha}$ is called a vector field on $A.$
A measurable field of Hilbert spaces over the indexing set $A$ is a field of
Hilbert spaces $\left\{  H_{\alpha}\right\}  _{\alpha\in A}$ together with a
countable set $\left\{  e_{j}\right\}  _{j}$ of vector fields such that

\begin{enumerate}
\item the functions $\alpha\mapsto\left\langle e_{j}\left(  \alpha\right)
,e_{k}\left(  \alpha\right)  \right\rangle _{H_{\alpha}}$ are measurable for
all $j,k,$

\item the linear span of $\left\{  e_{k}\left(  \alpha\right)  \right\}  _{k}$
is dense in $H_{\alpha}$ for each $\alpha\in A.$
\end{enumerate}

The direct integral of the spaces $H_{\alpha}$ with respect to the measure
$\mu$ is denoted by $\int_{A}^{\oplus}H_{\alpha}d\mu\left(  \alpha\right)  $
and is the space of measurable vector fields $\varphi$ on $A$ such that
$\int_{A}\left\Vert \varphi\left(  \alpha\right)  \right\Vert _{H_{\alpha}%
}^{2}d\mu\left(  \alpha\right)  <\infty.$ The inner product for this Hilbert
space is: $\left\langle \varphi_{1},\varphi_{2}\right\rangle =\int
_{A}\left\langle \varphi_{1}\left(  \alpha\right)  ,\varphi_{2}\left(
\alpha\right)  \right\rangle _{H_{\alpha}}d\mu\left(  \alpha\right)  \text{
for }\varphi_{1},\varphi_{2}\in\int_{A}^{\oplus}H_{\alpha}d\mu\left(
\alpha\right) .$

\section{The regular representation and its decompositions}

In this section, we will discuss the Plancherel theory for $\Gamma.$ For this
purpose, we will need a complete description of the unitary dual of $\Gamma.$
This will allow us to obtain a central decomposition of the left regular
representation of $\Gamma.$ Also, a proof of Proposition \ref{L} will be given
in this section.

Let $L$ be the left regular representation of $\Gamma.$ Suppose that $\Gamma$
is not commutative and that $B$ is a rational matrix. We have shown that
$\Gamma$ has an abelian normal subgroup of finite index which we call
$\Gamma_{0}$. Moreover, there is a canonical action ($74$-$79,$
\cite{Lipsman2}) of $\Gamma$ on the group $\widehat{\Gamma}_{0}$ (the unitary
dual of $\Gamma_{0})$ such that for each $P\in\Gamma$ and $\chi\in
\widehat{\Gamma}_{0},$ $P\cdot\chi\left(  Q\right)  =\chi\left(
P^{-1}QP\right) .$ Let us suppose that $\chi=\chi_{\left(  \lambda_{1}%
,\lambda_{2},\lambda_{3}\right)  }\ $is a character in the unitary dual
$\widehat{\Gamma}_{0}$ where $\left(  \lambda_{1},\lambda_{2},\lambda
_{3}\right)  \in\left\{  0,1,\cdots,m-1\right\}  \times\frac{\mathbb{R}^{d}%
}{B^{\star}\mathbb{Z}^{d}}\times\frac{\mathbb{R}^{d}}{A^{\star}\mathbb{Z}^{d}%
}\simeq\widehat{\Gamma}_{0},$ and $\chi_{\left(  \lambda_{1},\lambda
_{2},\lambda_{3}\right)  }\left(  e^{\frac{2\pi ik}{m}}M_{Bl}T_{Aj}\right)
=e^{\frac{2\pi i\lambda_{1}k}{m}}e^{2\pi i\left\langle \lambda_{2}%
,Bl\right\rangle }e^{2\pi i\left\langle \lambda_{3},Aj\right\rangle }.$
We observe that $\mathbb{R}^{d}$ is identified with its dual $\widehat
{\mathbb{R}^{d}}$ and $\left\{  0,1,\cdots,m-1\right\}  $ parametrizes the
unitary dual of the commutator subgroup $\left[  \Gamma,\Gamma\right]  $ which
is isomorphic to the cyclic group $\mathbb{Z}_{m}.$ For any $\tau\in\left[
\Gamma,\Gamma\right]  ,$ we have $\tau\cdot\chi_{\left(  \lambda_{1}%
,\lambda_{2},\lambda_{3}\right)  }=\chi_{\left(  \lambda_{1},\lambda
_{2},\lambda_{3}\right)  }.$ Moreover, given $M_{l},T_{k}\in\Gamma,$
\[
M_{l}\cdot\chi_{\left(  \lambda_{1},\lambda_{2},\lambda_{3}\right)  }%
=\chi_{\left(  \lambda_{1},\lambda_{2},\lambda_{3}\right)  },\text{ and }%
T_{k}\cdot\chi_{\left(  \lambda_{1},\lambda_{2},\lambda_{3}\right)  }%
=\chi_{\left(  \lambda_{1},\lambda_{2}-k\lambda_{1},\lambda_{3}\right)  }.
\]
Next, let $\Gamma_{\chi}=\left\{  P\in\Gamma:P\cdot\chi=\chi\right\}  .$ It is
easy to see that the stability subgroup of the character $\chi_{\left(
\lambda_{1},\lambda_{2},\lambda_{3}\right)  }$ is described as follows:
\begin{equation}
\Gamma_{\chi_{\left(  \lambda_{1},\lambda_{2},\lambda_{3}\right)  }}=\left\{
\tau M_{l}T_{k}\in\Gamma:\tau\in\left[  \Gamma,\Gamma\right]  ,l\in
B\mathbb{Z}^{d},k\lambda_{1}\in B^{\star}\mathbb{Z}^{d}\right\}  .
\label{stab}%
\end{equation}
It follows from (\ref{stab}) that the stability group $\Gamma_{\chi_{\left(
\lambda_{1},\lambda_{2},\lambda_{3}\right)  }}$ contains the normal subgroup
$\Gamma_{0}.$ Indeed, if $\lambda_{1}=0$ then $\Gamma_{\chi_{\left(
\lambda_{1},\lambda_{2},\lambda_{3}\right)  }}=\Gamma.$ Otherwise,
\begin{equation}
\Gamma_{\chi_{\left(  \lambda_{1},\lambda_{2},\lambda_{3}\right)  }}=\left\{
\tau M_{l}T_{k}\in\Gamma:\tau\in\left[  \Gamma,\Gamma\right]  ,\text{ }l\in
B\mathbb{Z}^{d},k\in\left(  \frac{1}{\lambda_{1}}B^{\star}\mathbb{Z}%
^{d}\right)  \cap\mathbb{Z}^{d}\right\}  . \label{stabilizer}%
\end{equation}
According to Mackey theory (see Page $76,$ \cite{Lipsman2}) and well-known
results of Kleppner and Lipsman (Page $460,$ \cite{Lipsman2}), every
irreducible representation of $\Gamma$ is of the type $\mathrm{Ind}%
_{\Gamma_{\chi}}^{\Gamma}\left(  \widetilde{\chi}\otimes\widetilde{\sigma
}\right) $ where $\widetilde{\chi}$ is an extension of a character $\chi$ of
$\Gamma_{0}$ to $\Gamma_{\chi},$ and $\widetilde{\sigma}$ is the lift of an
irreducible representation $\sigma$ of $\Gamma_{\chi}/\Gamma_{0}$ to
$\Gamma_{\chi}.$ Also two irreducible representations $\mathrm{Ind}%
_{\Gamma_{\chi}}^{\Gamma}\left(  \widetilde{\chi}\otimes\widetilde{\sigma
}\right)  \text{ and }\mathrm{Ind}_{\Gamma_{\chi_{1}}}^{\Gamma}\left(
\widetilde{\chi_{1}}\otimes\widetilde{\sigma}\right) $ are equivalent if and
only if the characters $\chi$ and $\chi_{1}$ of $\Gamma_{0}$ belong to the
same $\Gamma$-orbit. Since $\Gamma$ is a finite extension of its subgroup
$\Gamma_{0}$, then it is well known that there is a measurable set which is a
cross-section for the $\Gamma$-orbits in $\widehat{\Gamma}_{0}.$ Now, let
$\Sigma$ be a measurable subset of $\widehat{\Gamma}_{0}$ which is a
cross-section for the $\Gamma$-orbits in $\widehat{\Gamma}_{0}.$ The unitary
dual of $\Gamma$ is a fiber space which is described as follows:
\[
\widehat{\Gamma}={\displaystyle\bigcup\limits_{\chi\in\Sigma}}\left\{
\pi_{\chi,\sigma}=\mathrm{Ind}_{\Gamma_{\chi}}^{\Gamma}\left(  \widetilde
{\chi}\otimes\widetilde{\sigma}\right)  \text{ }:\text{ }\sigma\in
\widehat{\frac{\Gamma_{\chi}}{\Gamma_{0}}}\right\}  .
\]
Finally, since $\Gamma$ is a type $I$ group, there exists a unique standard
Borel measure on $\widehat{\Gamma}$ such that the left regular representation
of the group $\Gamma$ is equivalent to a direct integral of all elements in
the unitary dual of $\Gamma,$ and the multiplicity of each irreducible
representation occurring is equal to the dimension of the corresponding
Hilbert space. So, we obtain a decomposition of the representation $L$ into a
direct integral decomposition of its irreducible representations as follows
(see Theorem $3.31,$ \cite{hartmut} and Theorem $5.12,$ \cite{moussa})
\begin{equation}
L\simeq\int_{\Sigma}^{\oplus}\int_{\widehat{\frac{\Gamma_{\chi}}{\Gamma_{0}}}%
}^{\oplus}\oplus_{k=1}^{\mathrm{dim}\left( l^{2}\left(  \frac{\Gamma}%
{\Gamma_{\chi}}\right) \right) }\pi_{\chi,\sigma}d\omega_{\chi}\left(
\sigma\right)  d\chi\label{Left}%
\end{equation}
and $\dim\left(  l^{2}\left(  \frac{\Gamma}{\Gamma_{\chi}}\right)  \right)
\leq\mathrm{card}\left(  \Gamma/\Gamma_{0}\right)  .$ The fact that
$\dim\left(  l^{2}\left(  \frac{\Gamma}{\Gamma_{\chi}}\right)  \right)
\leq\mathrm{card}\left(  \Gamma/\Gamma_{0}\right)  $ is justified because the
number of representative elements of the quotient group $\frac{\Gamma}%
{\Gamma_{\chi}}$ is bounded above by the number of elements in $\frac{\Gamma
}{\Gamma_{0}}.$ Moreover the direct integral representation in (\ref{Left}) is
realized as acting in the Hilbert space
\begin{equation}
\int_{\Sigma}^{\oplus}\int_{\widehat{\frac{\Gamma_{\chi}}{\Gamma_{0}}}%
}^{\oplus}\oplus_{k=1}^{\mathrm{dim}\left( l^{2}\left(  \frac{\Gamma}%
{\Gamma_{\chi}}\right)  \right) }l^{2}\left(  \frac{\Gamma}{\Gamma_{\chi}%
}\right) d\omega_{\chi}\left(  \sigma\right)  d\chi. \label{canonL}%
\end{equation}

Although the decomposition in (\ref{Left}) is canonical, the decomposition
provided by Proposition \ref{L} will be more convenient for us.

\subsection{Proof of Proposition \ref{L}}
Let $e$ be the identity element in $\Gamma,$ and let $1$ be the trivial
representation of the trivial group $\left\{  e\right\}  .$ We observe that
$L\simeq\mathrm{Ind}_{\left\{  e\right\}  }^{\Gamma}\left(  1\right)  .$ It
follows that
\begin{align}
L  &  \simeq\mathrm{Ind}_{\Gamma_{0}}^{\Gamma}\left(  \mathrm{Ind}_{\left\{
e\right\}  }^{\Gamma_{0}}\left(  1\right)  \right) \label{eq}\\
&  \simeq\mathrm{Ind}_{\Gamma_{0}}^{\Gamma}\left(  \int_{\widehat{\Gamma}_{0}%
}^{\oplus}\chi_{t}\text{ }dt\right) \nonumber\\
&  \simeq\int_{\widehat{\Gamma}_{0}}^{\oplus}\mathrm{Ind}_{\Gamma_{0}}%
^{\Gamma}\left(  \chi_{t}\right)  \text{ }dt. \label{q}%
\end{align}
The second equivalence given above is coming from the fact that $\mathrm{Ind}%
_{\left\{  e\right\}  }^{\Gamma_{0}}\left(  1\right)  $ is equivalent to the
left regular representation of the group $\Gamma_{0}.$ Since $\Gamma_{0}$ is
abelian, its left regular representation admits a direct integral
decomposition into elements in the unitary dual of $\Gamma_{0},$ each
occurring once. Moreover, the measure $dt$ is a Lebesgue measure (also a Haar
measure) supported on the unitary dual of the group, and the Plancherel
transform is the unitary operator which is intertwining the representations
$\mathrm{Ind}_{\left\{  e\right\}  }^{\Gamma_{0}}\left(  1\right)  \text{ and
}\int_{\widehat{\Gamma}_{0}}^{\oplus}\chi_{t}\text{ }dt.$ Based on the
discussion above, it is worth mentioning that the representations occurring in
(\ref{q}) are generally reducible since it is not always the case that
$\Gamma_{0}=\Gamma_{\chi_{t}}.$ We observe that $\widehat{\Gamma}_{0}$ is
parametrized by the group $\mathbb{Z}_{m}\times\frac{\mathbb{R}^{d}}{B^{\star
}\mathbb{Z}^{d}}\times\frac{\mathbb{R}^{d}}{A^{\star}\mathbb{Z}^{d}}.$ Thus,
identifying $\widehat{\Gamma}_{0}$ with $\mathbb{Z}_{m}\times\frac
{\mathbb{R}^{d}}{B^{\star}\mathbb{Z}^{d}}\times\frac{\mathbb{R}^{d}}{A^{\star
}\mathbb{Z}^{d}},$ we reach the desired result: $L\simeq\oplus_{k=0}^{m-1}%
\int_{\frac{\mathbb{R}^{d}}{B^{\star}\mathbb{Z}^{d}}\times\frac{\mathbb{R}%
^{d}}{A^{\star}\mathbb{Z}^{d}}}^{\oplus}\mathrm{Ind}_{\Gamma_{0}}^{\Gamma}%
\chi_{\left(  k,t\right)  }\text{ }dt.$

\begin{remark}
\label{indirreducible}Referring to (\ref{stabilizer}), we remark that
$\Gamma_{\chi_{\left(  1,t_{2},t_{3}\right)  }}=\Gamma_{0},$ and in this case
$\mathrm{Ind}_{\Gamma_{0}}^{\Gamma}\left(  \chi_{\left(  1,t_{2},t_{3}\right)
}\right)  $ is an irreducible representation of the group $\Gamma.$
\end{remark}

\section{Decomposition of $\pi$}

In this section, we will provide a decomposition of the Gabor representation
$\pi.$ For this purpose, it is convenient to regard the set $\mathbb{R}^{d}$
as a fiber space, with base space the $d$-dimensional torus. Next, for any
element $t$ in the torus, the corresponding fiber is the set $t+\mathbb{Z}%
^{d}.$ With this concept in mind, let us define the periodization map
$\mathfrak{R}:L^{2}\left(  \mathbb{R}^{d}\right)  \rightarrow\int
_{\frac{\mathbb{R}^{d}}{\mathbb{Z}^{d}}}^{\oplus}l^{2}\left(  \mathbb{Z}%
^{d}\right)  dt$ such that $\mathfrak{R}f\left(  t\right)  =\left(  f\left(
t+k\right)  \right)  _{k\in\mathbb{Z}^{d}}.$ We remark here that we clearly
abuse notation by making no distinction between $\frac{\mathbb{R}^{d}%
}{\mathbb{Z}^{d}}$ and a choice of a measurable subset of $\mathbb{R}^{d}$ which is a fundamental domain for $\frac{\mathbb{R}%
^{d}}{\mathbb{Z}^{d}}.$ Next, the inner product which we endow the direct
integral Hilbert space $\int_{\frac{\mathbb{R}^{d}}{\mathbb{Z}^{d}}}^{\oplus
}l^{2}\left(  \mathbb{Z}^{d}\right)  dt$ with is defined as follows. For any
vectors
\[
f\text{ and }h\in\int_{\frac{\mathbb{R}^{d}}{\mathbb{Z}^{d}}}^{\oplus}%
l^{2}\left(  \mathbb{Z}^{d}\right)  dt
\]
the inner product of $f$ and $g$ is equal to $\left\langle f,h\right\rangle
_{\mathfrak{R}\left(  L^{2}\left(  \mathbb{R}^{d}\right)  \right)  }%
=\int_{\frac{\mathbb{R}^{d}}{\mathbb{Z}^{d}}}\left(  \sum_{k\in\mathbb{Z}^{d}%
}f\left(  t+k\right)  \overline{h\left(  t+k\right)  }\right)  \text{ }dt,$
and it is easy to check that $\mathfrak{R}$ is a unitary map.

\subsection{Proof of Proposition \ref{operator}}

For $t\in\mathbb{R}^{d},$ we consider the unitary character $\chi_{\left(
1,-t\right)  }:\Gamma_{1}\rightarrow\mathbb{T}$ which is defined by
$\chi_{\left(  1,-t\right)  }\left(  e^{2\pi iz}M_{l}\right)  =e^{2\pi
iz}e^{-2\pi i\left\langle t,l\right\rangle }.$ Next, we compute the action of
the unitary representation $\mathrm{Ind}_{\Gamma_{1}}^{\Gamma}\chi_{\left(
1,-t\right)  }$ of $\Gamma$ which is acting in the Hilbert space
\[
\mathbf{H}_{t}=\left\{
\begin{array}
[c]{c}%
f:\Gamma\rightarrow\mathbb{C}:f\left(  PQ\right)  =\chi_{\left(  1,-t\right)
}\left(  Q\right)  ^{-1}f\left(  P\right)  ,Q\in\Gamma_{1}\text{ }\\
\text{and }\sum_{P\Gamma_{1}\in\frac{\Gamma}{\Gamma_{1}}}\left\vert f\left(
P\right)  \right\vert ^{2}<\infty
\end{array}
\right\}  .
\]
Let $\Theta$ be a cross-section for $\Gamma/\Gamma_{1}$ in $\Gamma.$ The
Hilbert space $\mathbf{H}_{t}$ is naturally identified with $l^{2}\left(
\Theta\right)  $ since for any $Q\in\Gamma_{1},$ we have $\left\vert f\left(
PQ\right)  \right\vert =\left\vert f\left(  P\right)  \right\vert .$ Via this
identification, we may realize the representation $\mathrm{Ind}_{\Gamma_{1}%
}^{\Gamma}\chi_{\left(  1,-t\right)  }$ as acting in $l^{2}\left(
\Theta\right)  .$ More precisely, for $a\in l^{2}\left(  \Theta\right) $ and
$\rho_{t}=\mathrm{Ind}_{\Gamma_{1}}^{\Gamma}\chi_{\left(  1,-t\right)  }$ we
have
\[
\left(  \rho_{t}\left(  X\right)  a\right)  \left(  T_{j}\right)  =\left\{
\begin{array}
[c]{cc}%
a\left(  T_{j-k}\right)  & \text{ if }X=T_{k}\\
e^{-2\pi i\left\langle j,l\right\rangle }e^{-2\pi i\left\langle
t,l\right\rangle }a\left(  T_{j}\right)  & \text{if }X=M_{l}\\
e^{2\pi i\theta}a\left(  T_{j}\right)  & \text{if }X=e^{2\pi i\theta}%
\end{array}
\right.  .
\]
Defining the unitary operator $\mathfrak{J:}$ $l^{2}\left(  \Theta\right)
\rightarrow l^{2}\left(  \mathbb{Z}^{d}\right)  $ such that $\left(
\mathfrak{J}a\right)  \left(  j\right)  =a\left(  T_{j}\right)  ,$ it is
easily checked that:%
\[
\mathfrak{J}^{-1}\left[  \left(  \mathfrak{R}Xf\right)  \left(  t\right)
\right]  =\left\{
\begin{array}
[c]{cc}%
\rho_{t}\left(  T_{k}\right)  \left[  \mathfrak{J}^{-1}\left(  \mathfrak{R}%
f\left(  t\right)  \right)  \right]  & \text{ if }X=T_{k}\\
\rho_{t}\left(  M_{l}\right)  \left[  \mathfrak{J}^{-1}\left(  \mathfrak{R}%
f\left(  t\right)  \right)  \right]  & \text{if }X=M_{l}\\
\rho_{t}\left(  e^{2\pi i\theta}\right)  \left[  \mathfrak{J}^{-1}\left(
\mathfrak{R}f\left(  t\right)  \right)  \right]  & \text{if }X=e^{2\pi
i\theta}%
\end{array}
\right.  .
\]
Thus, the representation $\pi$ is unitarily equivalent to
\begin{equation}
\int_{\frac{\mathbb{R}^{d}}{\mathbb{Z}^{d}}}^{\oplus}\rho_{t}\text{ }dt.
\label{Part 1}%
\end{equation}
Now, we remark that $\rho_{t}$ is not an irreducible representation of the
group $\Gamma$. Indeed, by inducing in stages (see Page $166,$ \cite{Folland}%
), we obtain that the representation $\rho_{t}$ is unitarily equivalent to
$\mathrm{Ind}_{\Gamma_{0}}^{\Gamma}\left(  \mathrm{Ind}_{\Gamma_{1}}%
^{\Gamma_{0}}\left(  \chi_{\left(  1,-t\right)  }\right)  \right)  $ and
$\mathrm{Ind}_{\Gamma_{1}}^{\Gamma_{0}}\chi_{\left(  1,-t\right)  }$ acts in
the Hilbert space
\begin{equation}
\mathbf{K}_{t}=\left\{
\begin{array}
[c]{c}%
f:\Gamma_{0}\rightarrow\mathbb{C}:f\left(  PQ\right)  =\chi_{\left(
1,-t\right)  }\left(  Q^{-1}\right)  f\left(  P\right)  ,Q\in\Gamma_{1}\text{
}\\
\text{and }\sum_{P\Gamma_{1}\in\frac{\Gamma_{0}}{\Gamma_{1}}}\left\vert
f\left(  P\right)  \right\vert ^{2}<\infty
\end{array}
\right\}  .
\end{equation}
Since $\chi_{(1,-t)}$ is a character, for any $f\in\mathbf{K}_{t},$ we notice
that $\left\vert f\left(  PQ\right)  \right\vert =\left\vert f\left(
P\right)  \right\vert $ for $Q\in\Gamma_{1}.$ Thus, the Hilbert space
$\mathbf{K}_{t}$ is naturally identified with $l^{2}\left(  \frac{\Gamma_{0}%
}{\Gamma_{1}}\right)  \simeq l^{2}\left(  A\mathbb{Z}^{d}\right)  $ where
$A\mathbb{Z}^{d}$ is a parametrizing set for the quotient group $\frac
{\Gamma_{0}}{\Gamma_{1}}.$ Via this identification, we may realize
$\mathrm{Ind}_{\Gamma_{1}}^{\Gamma_{0}}\chi_{\left(  1,-t\right)  }$ as acting
in $l^{2}\left(  A\mathbb{Z}^{d}\right)  .$ More precisely, for $T_{j},j\in
A\mathbb{Z}^{d},$ we compute the action of $\mathrm{Ind}_{\Gamma_{1}}%
^{\Gamma_{0}}\chi_{\left(  1,-t\right)  }$ as follows:%
\[
\left[  \mathrm{Ind}_{\Gamma_{1}}^{\Gamma_{0}}\chi_{\left(  1,-t\right)
}\left(  X\right)  f\right]  \left(  T_{j}\right)  =\left\{
\begin{array}
[c]{cc}%
f\left(  T_{j-k}\right)  & \text{ if }X=T_{k}\\
e^{-2\pi i\left\langle t,l\right\rangle }f\left(  T_{j}\right)  & \text{if
}X=M_{l}\\
e^{2\pi i\theta}f\left(  T_{j}\right)  & \text{if }X=e^{2\pi i\theta}%
\end{array}
\right.  .
\]
Now, let $\mathbf{F}_{A\mathbb{Z}^{d}}$ be the Fourier transform defined on
$l^{2}\left(  A\mathbb{Z}^{d}\right)  .$ Given a vector $f$ in $l^{2}\left(
A\mathbb{Z}^{d}\right)  ,$ it is not too hard to see that%
\[
\left[  \mathbf{F}_{A\mathbb{Z}^{d}}\left(  \mathrm{Ind}_{\Gamma_{1}}%
^{\Gamma_{0}}\chi_{\left(  1,-t\right)  }\left(  X\right)  f\right)  \right]
\left(  \xi\right)  =\left\{
\begin{array}
[c]{cc}%
\chi_{\xi}\left(  T_{k}\right)  \mathbf{F}_{A\mathbb{Z}^{d}}f\left(
\xi\right)  & \text{ if }X=T_{k}\\
e^{-2\pi i\left\langle t,l\right\rangle }\mathbf{F}_{A\mathbb{Z}^{d}}f\left(
\xi\right)  & \text{if }X=M_{l}\\
e^{2\pi i\theta}\mathbf{F}_{A\mathbb{Z}^{d}}f\left(  \xi\right)  & \text{if
}X=e^{2\pi i\theta}%
\end{array}
\right.
\]
where $\chi_{\xi}$ is a character of the discrete group $A \mathbb{Z}^{d}.$ As
a result, given $X\in\Gamma$ we obtain:%
\begin{equation}
\mathbf{F}_{A\mathbb{Z}^{d}}\circ\rho_{t}\left(  X\right)  \circ
\mathbf{F}_{A\mathbb{Z}^{d}}^{-1}=\int_{\frac{\mathbb{R}^{d}}{A^{\star
}\mathbb{Z}^{d}}}^{\oplus}\chi_{\left(  1,-t,\xi\right)  }\left(  X\right)
\text{ }d\xi, \label{induced}%
\end{equation}
where $\chi_{\left(  1,-t,\xi\right)  }$ is a character of $\Gamma_{0}$
defined as follows:%
\[
\chi_{\left(  1,-t,\xi\right)  }\left(  X\right)  =\left\{
\begin{array}
[c]{cc}%
\chi_{\xi}\left(  T_{k}\right)  & \text{ if }X=T_{k}\\
e^{-2\pi i\left\langle t,l\right\rangle } & \text{if }X=M_{l}\\
e^{2\pi i\theta} & \text{if }X=e^{2\pi i\theta}%
\end{array}
\right.  .
\]
Using the fact that induction commutes with direct integral decomposition (see Page $41,$\cite{Corwin}) we have
\begin{equation}
\rho_{t}\simeq\mathrm{Ind}_{\Gamma_{0}}^{\Gamma}\left(  \int_{\frac
{\mathbb{R}^{d}}{A^{\star}\mathbb{Z}^{d}}}^{\oplus}\chi_{\left(
1,-t,\xi\right)  }\text{ }d\xi\right)  \simeq\int_{\frac{\mathbb{R}{d}%
}{A^{\star}\mathbb{Z}^{d}}}^{\oplus}\left(  \mathrm{Ind}_{\Gamma_{0}}^{\Gamma
}\chi_{\left(  1,-t,\xi\right)  }\right)  \text{ }d\xi. \label{Part 2}%
\end{equation}
Putting (\ref{Part 1}) and (\ref{Part 2}) together, we arrive at: $\pi
\simeq\int_{\frac{\mathbb{R}^{d}}{\mathbb{Z}^{d}}\times\frac{\mathbb{R}^{d}%
}{A^{\star}\mathbb{Z}^{d}}}^{\oplus}\mathrm{Ind}_{\Gamma_{0}}^{\Gamma}%
\chi_{\left(  1,\sigma\right)  }d\sigma.$ Finally, the fact that
$\mathrm{Ind}_{\Gamma_{0}}^{\Gamma}\chi_{\left(  1,\sigma\right)  }$ is an
irreducible representation is due to Remark \ref{indirreducible}. This
completes the proof.
\begin{remark}
We shall offer here a different proof of Proposition \ref{operator} by exhibiting an explicit intertwining operator which is a version of
the Zak transform for the representation $\pi$ and the direct integral
representation described given in (\ref{comp}).  Let $C_{c}\left(\mathbb{R}^{d}\right)  $ be the space of continuous functions on $\mathbb{R}^{d}$ which are compactly supported. Let $\mathcal{Z}$ be the operator which
maps each $f\in C_{c}\left(\mathbb{R}^{d}\right)$ to the function
\begin{equation}
\left(  \mathcal{Z}f\right)  \left(  x,w,j+A\mathbb{Z}^{d}\right)  =\sum_{k\in\mathbb{Z}^{d}}f\left(  x+Ak+j\right)  e^{2\pi i\left\langle w,Ak\right\rangle }%
\equiv\phi\left(  x,w,j+A\mathbb{Z}^{d}\right)  \label{Zack}
\end{equation}
where $x,w\in\mathbb{R}^{d}$ and $j$ is an element of a cross-section for $\frac{\mathbb{Z}^{d}}{A\mathbb{Z}^{d}}$ in the lattice $\mathbb{Z}^{d}.$ Given arbitrary $m^{\prime}\in\mathbb{Z}^{d},$ we may write $m^{\prime}=Ak^{\prime}+j^{\prime}$ where $k^{\prime}\in\mathbb{Z}^{d}$ and $j^{\prime}$ is an element of a cross-section for $\frac{\mathbb{Z}^{d}}{A\mathbb{Z}^{d}}.$ Next, it is worth observing that given $m,m^{\prime}\in\mathbb{Z}^{d},$ $
\phi\left(  x,w+A^{\star}m,j+A\mathbb{Z}^{d}\right) =\phi\left(  x,w,j+A\mathbb{Z}^{d}\right)$ and $
\phi\left(  x+m^{\prime},w,j+A\mathbb{Z}^{d}\right)$ is equal to $e^{-2\pi i\left\langle w,Ak^{\prime}\right\rangle }
\phi\left(  x,w,j+j^{\prime}+A\mathbb{Z}^{d}\right).$
This observation will later help us explain the meaning of Equality
(\ref{Zack}). Let $\Sigma_{1}$ and $\Sigma_{2}$ be measurable cross-sections
for $\frac{\mathbb{R}^{d}}{\mathbb{Z}^{d}}$ and $\frac{\mathbb{R}^{d}}{A^{\star}\mathbb{Z}^{d}}$ respectively in $\mathbb{R}^{d}$. For example, we may pick $\Sigma_{1}=\left[  0,1\right)  ^{d}$ and
$\Sigma_{2}=A^{\star}\left[  0,1\right)  ^{d}.$ Since $f$ is
square-integrable, by a periodization argument it is easy to see that
\begin{equation}
\left\Vert f\right\Vert _{L^{2}\left(\mathbb{R}^{d}\right)  }^{2}=\int_{\Sigma_{1}}\sum_{m\in\mathbb{Z}^{d}}\left\vert f\left(  x+m\right)  \right\vert ^{2}dx.\label{finite}%
\end{equation}
Therefore, the integral on the right of (\ref{finite}) is finite. It
immediately follows that
\[
\int_{\Sigma_{1}}\sum_{j+A\mathbb{Z}^{d}\in\mathbb{Z}^{d}/A\mathbb{Z}^{d}}\sum_{k\in\mathbb{Z}^{d}}\left\vert f\left(  x+Ak+j\right)  \right\vert ^{2}dt<\infty.
\]
Therefore, the sum $\sum_{j+A\mathbb{Z}^{d}\in\mathbb{Z}^{d}/A\mathbb{Z}^{d}}\sum_{k\in\mathbb{Z}^{d}}\left\vert f\left(  x+Ak+j\right)  \right\vert ^{2}$ is finite for almost
every $x\in\Sigma_{1}$ and a fixed $j$ which is a cross-section for $\mathbb{Z}^{d}/A\mathbb{Z}^{d}$ in $\mathbb{Z}^{d}.$ Next, observe that
\begin{equation}
\sum_{k\in\mathbb{Z}^{d}}f\left(  x+Ak+j\right)  e^{2\pi i\left\langle w,Ak\right\rangle
}\label{series}%
\end{equation}
is a Fourier series of the sequence $\left(  f\left(  x+Ak+j\right)  \right)
_{Ak\in A\mathbb{Z}^{d}}\in l^{2}\left(  A\mathbb{Z}^{d}\right)  .$ So, for almost every $x$ and for a fixed $j,$ the function
$\phi\left(  x,\cdot,j+A\mathbb{Z}^{d}\right)  $ is regarded as a function of $L^{2}\left(\mathbb{R}^{d}\right)  $ which is $A^{\star}\mathbb{Z}^{d}$-periodic (it is an $L^{2}\left(  \Sigma_{2}\right)  $ function). In
summary, we may regard the function $\left(  \mathcal{Z}f\right)  \left(x,w,j+A\mathbb{Z}^{d}\right)  $ as being defined over the set $\Sigma_{1}\times\Sigma_{2}\times\frac{\mathbb{Z}^{d}}{A\mathbb{Z}
^{d}}.$ Let us now show that $\mathcal{Z}$ maps $C_{c}\left(\mathbb{R}^{d}\right)  $ isometrically into the Hilbert space $L^{2}\left(  \Sigma
_{1}\times\Sigma_{2}\times\frac{\mathbb{Z}^{d}}{A\mathbb{Z}^{d}}\right).$ Given any square-summable function $f$ in $L^{2}\left(\mathbb{R}^{d}\right),$ we have
\[
\int_{\mathbb{R}^{d}}\left\vert f\left(  t\right)  \right\vert ^{2}dt=\int_{\Sigma_{1}}
\sum_{k\in\mathbb{Z}^{d}}\left\vert f\left(  t+k\right)  \right\vert ^{2}dt=\int_{\Sigma_{1}}
\sum_{j+A\mathbb{Z}^{d}\in\mathbb{Z}^{d}/A\mathbb{Z}^{d}}\sum_{k\in\mathbb{Z}^{d}}\left\vert f\left(  t+Ak+j\right)  \right\vert ^{2}dt.
\]
Regarding $\left(  f\left(  t+Ak+j\right)  \right)  _{Ak\in A\mathbb{Z}^{d}}$ as a square-summable sequence in $l^{2}\left(  A\mathbb{Z}^{d}\right),$ the function
\[
w\mapsto\sum_{k\in\mathbb{Z}^{d}}f\left(  t+Ak+j\right)  e^{2\pi i\left\langle w,Ak\right\rangle }
\]
is the Fourier transform of the sequence $\left(  f\left(  t+Ak+j\right)
\right)  _{Ak\in A\mathbb{Z}^{d}}$. The Plancherel theorem being a unitary operator, we have
\[
\sum_{k\in\mathbb{Z}^{d}}\left\vert f\left(  t+Ak+j\right)  \right\vert ^{2}=\int_{\Sigma_{2}}\left\vert \sum_{k\in\mathbb{Z}^{d}}f\left(  t+Ak+j\right)  e^{2\pi i\left\langle w,Ak\right\rangle
}\right\vert ^{2}dw.
\]
It follows that
\begin{align*}
\int_{\mathbb{R}^{d}}\left\vert f\left(  t\right)  \right\vert ^{2}dt &  =\int_{\Sigma_{1}%
}\sum_{j+A\mathbb{Z}^{d}\in\mathbb{Z}^{d}/A\mathbb{Z}^{d}}\int_{\Sigma_{2}}\left\vert \sum_{k\in\mathbb{Z}^{d}}f\left(  t+Ak+j\right)  e^{2\pi i\left\langle w,Ak\right\rangle
}\right\vert ^{2}dwdt\\&  =\int_{\Sigma_{1}}\int_{\Sigma_{2}}\sum_{j+A\mathbb{Z}^{d}\in\mathbb{Z}^{d}/A\mathbb{Z}^{d}}\left\vert \mathcal{Z}f\left(  x,w,j+A\mathbb{Z}^{d}\right)  \right\vert ^{2}dwdt\\&  =\left\Vert \mathcal{Z}f\right\Vert _{L^{2}\left(  \Sigma_{1}\times\Sigma_{2}\times\frac{\mathbb{Z}^{d}}{A\mathbb{Z}^{d}}\right)  }^{2}.
\end{align*}
Now, by density, we may extend the operator $\mathcal{Z}$ to $L^{2}\left(\mathbb{R}^{d}\right),$ and we shall next show that the extension
\[
\mathcal{Z}:L^{2}\left(\mathbb{R}^{d}\right)  \rightarrow L^{2}\left(  \Sigma_{1}\times\Sigma_{2}\times\frac{\mathbb{Z}^{d}}{A\mathbb{Z}^{d}}\right)
\]
is unitary. At this point, we only need to show that $\mathcal{Z}$ is
surjective. Let $\varphi$ be any vector in the Hilbert space $L^{2}\left(\Sigma_{1}\times\Sigma_{2}\times\frac{\mathbb{Z}^{d}}{A\mathbb{Z}^{d}}\right).$ Clearly for almost every $x$ and given any fixed $j,$ we have
$\varphi\left(  x,\cdot,j\right)  \in L^{2}\left(  \Sigma_{2}\right)  .$ For
such $x$ and $j,$ let $\left(  c_{\ell}\left(  x,j\right)  \right)  _{\ell\in
A\mathbb{Z}^{d}}$ be the Fourier transform of $\varphi\left(  x,\cdot,j\right)  .$ Next,
define $f_{\varphi}\in L^{2}\left(\mathbb{R}^{d}\right)  $ such that for almost every $x\in\Sigma_{1},$
\[
f_{\varphi}\left(  x+A\ell+j\right)  =c_{\ell}\left(  x,j\right)  .
\]
Now, for almost every $w\in\Sigma_{2},$
\begin{align*}
Zf_{\varphi}\left(  x,w\right)   &  =\sum_{\ell\in\mathbb{Z}^{d}}f_{\varphi}\left(  x+A\ell+j\right)  e^{2\pi i\left\langle w,A\ell
\right\rangle }\\
&  =\sum_{\ell\in\mathbb{Z}^{d}}c_{\ell}\left(  x,j\right)  e^{2\pi i\left\langle w,A\ell\right\rangle
}\\
&  =\varphi\left(  x,w,j+A\mathbb{Z}^{d}\right).
\end{align*}
It remains to show that our version of Zak transform intertwines the
representation $\pi$ with $\int_{\Sigma_{1}\times\Sigma_{2}}^{\oplus}%
\rho_{\left(  1,x,w\right)  }$ $dxdw$ such that $\rho_{\left(  1,x,w\right)
}$ is equivalent to the induced representation $\mathrm{Ind}_{\Gamma_{0}%
}^{\Gamma}\chi_{\left(  1,-x,w\right)  }$. Let $\mathcal{R}:l^{2}\left(
\Gamma/\Gamma_{0}\right)  \rightarrow l^{2}\left(\mathbb{Z}^{d}/A\mathbb{Z}^{d}\right)$ be a unitary map defined by
\[
\mathcal{R}\left(  f\left(  j+A\mathbb{Z}^{d}\right)  _{j+A\mathbb{Z}^{d}}\right)  =\left(  f\left(  T_{j}\Gamma_{0}\right)  \right)  _{T_{j}\Gamma_{0}}.
\]
Put
\[
\rho_{\left(  1,x,w\right)  }\left(  X\right)  =\mathcal{R}\circ
\mathrm{Ind}_{\Gamma_{0}}^{\Gamma}\chi_{\left(  1,-x,w\right)  }\left(
X\right)  \circ\mathcal{R}^{-1}\text{ for every }X\in\Gamma.
\]
It is straightforward to check that
\begin{align*}
\left(  \mathcal{Z}T_{j}f\right)  \left(  x,w,\cdot\right)   &  =\sum_{k\in\mathbb{Z}^{d}}T_{j}f\left(  x+\cdot\right)  e^{2\pi i\left\langle w,Ak\right\rangle }\\
&  =\sum_{k\in\mathbb{Z}^{d}}T_{j}f\left(  x+\left(  \cdot-j\right)  \right)  e^{2\pi i\left\langle
w,Ak\right\rangle }\\
&  =\left[  \rho_{\left(  1,x,w\right)  }\left(  T_{j}\right)  \right]
\left(  \mathcal{Z}f\right)  \left(  x,w,\cdot\right)
\end{align*}
and
\begin{align*}
\left(  \mathcal{Z}M_{Bl}f\right)  \left(  x,w,\cdot\right)   &  =\sum_{k\in\mathbb{Z}^{d}}e^{-2\pi i\left\langle Bl,x+Ak+j\right\rangle }f\left(  x+\cdot\right)
e^{2\pi i\left\langle w,Ak\right\rangle }\\
&  =e^{-2\pi i\left\langle Bl,x+j\right\rangle }\sum_{k\in\mathbb{Z}^{d}}e^{-2\pi i\left\langle Bl,Ak\right\rangle }f\left(  x+\cdot\right)
e^{2\pi i\left\langle w,Ak\right\rangle }\\
&  =e^{-2\pi i\left\langle Bl,x+j\right\rangle }f\left(  x+\cdot\right)
e^{2\pi i\left\langle w,Ak\right\rangle }\\
&  =e^{-2\pi i\left\langle Bl,x+j\right\rangle }\left(  \mathcal{Z}f\right)
\left(  x,w,\cdot\right)  \\
&  =\rho_{\left(  1,x,w\right)  }\left(  M_{Bl}\right)  \left(  \mathcal{Z}%
f\right)  \left(  x,w,\cdot\right)  .
\end{align*}
In summary, given any $X\in\Gamma,$
\[
\mathcal{Z\circ\pi}\left(  X\right)  \mathcal{\circ Z}^{-1}=\int_{\Sigma
_{1}\times\Sigma_{2}}^{\oplus}\rho_{\left(  1,x,w\right)  }\left(  X\right)
\text{ }dxdw.
\]
\end{remark}
\begin{lemma}
\label{hard} Let $\Gamma_{1}=A_{1}\mathbb{Z}^{d}$ and $\Gamma_{2}%
=A_{2}\mathbb{Z}^{d}$ be two lattices of $\mathbb{R}^{d}$ such that $A_{1}$
and $A_{2}$ are non-singular matrices and $\left\vert \det A_{1}\right\vert
\leq\left\vert \det A_{2}\right\vert .$ Then there exist measurable sets
$\Sigma_{1},\Sigma_{2}$ such that $\Sigma_{1}$ is a fundamental domain for
$\frac{\mathbb{R}^{d}}{A_{1}\mathbb{Z}^{d}}$ and $\Sigma_{2}$ is a fundamental
domain for $\frac{\mathbb{R}^{d}}{A_{2}\mathbb{Z}^{d}}$ and $\Sigma
_{1}\subseteq\Sigma_{2}\subset\mathbb{R}^{d}.$
\end{lemma}

\begin{proof}
According to Theorem $1.2,$ \cite{Han}, there exists a measurable set
$\Sigma_{1}$ such that $\Sigma_{1}$ tiles $\mathbb{R}^{d}$ by the lattice
$A_{1}\mathbb{Z}^{d}$ and packs $\mathbb{R}^{d}$ by $A_{2}\mathbb{Z}^{d}$. By
packing, we mean that given any distinct $\gamma,\kappa\in A_{2}\mathbb{Z}%
^{d}$, the set $\left(  \Sigma_{1}+\gamma\right)  \cap\left(  \Sigma
_{1}+\kappa\right)  $ is an empty set and $\sum_{\lambda\in A_{2}%
\mathbb{Z}^{d}}1_{\Sigma_{1}}\left(  x+\lambda\right)  \leq1\text{ for }%
x\in\mathbb{R}^{d}$ where $1_{\Sigma_{1}}$ denotes the characteristic function
of the set $\Sigma_{1}.$ We would like to construct a set $\Sigma_{2}$ which
tiles $\mathbb{R}^{d}$ by $A_{2}\mathbb{Z}^{d}$ such that $\Sigma_{1}%
\subseteq\Sigma_{2}.$ To construct such a set, let $\Omega$ be a fundamental
domain for $\frac{\mathbb{R}^{d}}{A_{2}\mathbb{Z}^{d}}.$ It follows that,
there exists a subset $I$ of $A_{2}\mathbb{Z}^{d}$ such that $\Sigma
_{1}\subseteq\text{ }\overset{\cdot}{{\displaystyle\bigcup\limits_{k\in I}}%
}\left(  \Omega+k\right) $ and each $\left(  \Omega+k\right)  \cap\Sigma_{1}$
is a set of positive Lebesgue measure. Next, for each $k\in I$, we define
$\Omega_{k}=\left(  \Omega+k\right)  \cap\Sigma_{1}.$ We observe that
$\Sigma_{1}=\text{ }\overset{\cdot}{{\displaystyle\bigcup\limits_{k\in I}}%
}\left(  \left(  \Omega+k\right)  \cap\Sigma_{1}\right)  =\text{ }%
\overset{\cdot}{{\displaystyle\bigcup\limits_{k\in I}}}\Omega_{k}$ where
$\Omega_{k}=(\Omega+k)\cap\Sigma_{1}.$ Put $$\Sigma_{2}=\left(  \Omega-\text{
}\overset{\cdot}{{\displaystyle\bigcup\limits_{k\in I}}}\left(  \Omega
_{k}-k\right)  \right)  \overset{\cdot}{\cup}\Sigma_{1}.$$ The disjoint union
in the equality above is due to the fact that for distinct $k,j\in I,$ the set
$\left(  \Omega_{k}-k\right)  \cap\left(  \Omega_{j}-j\right)  $ is a null
set. This holds because, $\Sigma_{1}$ packs $\mathbb{R}^{d}$ by $A_{2}%
\mathbb{Z}^{d}.$ Finally, we observe that $$\Sigma_{2}=\left(  \Omega-\text{
}\overset{\cdot}{{\displaystyle\bigcup\limits_{k\in I}} }\left(  \Omega
_{k}-k\right)  \right)  \overset{\cdot}{\cup}\left(  \overset{\cdot
}{{\displaystyle\bigcup\limits_{k\in I}}}\Omega_{k}\right) $$ and
$\Omega=\left(  \Omega-\text{ }\overset{\cdot}{{\displaystyle\bigcup
\limits_{k\in I}} }\left(  \Omega_{k}-k\right)  \right)  \overset{\cdot}{\cup
}\left(  \overset{\cdot}{{\displaystyle\bigcup\limits_{k\in I}}}\Omega
_{k}^{\prime}\right) $ where each $\Omega_{k}^{\prime}$ is $A_{2}%
\mathbb{Z}^{d}$-congruent with $\Omega_{k}.$ Therefore $\Sigma_{2}$ is a
fundamental domain for $\frac{\mathbb{R}^{d}}{A_{2}\mathbb{Z}^{d}}$ which
contains $\Sigma_{1}.$ This completes the proof.
\end{proof}

Now, we are ready to prove Proposition \ref{pi}. Part of the proof of
Proposition \ref{pi} relies on some technical facts related to central
decompositions of unitary representations. A good presentation of this theory
is found in Section $3.4.2,$ \cite{hartmut}.

\subsection{Proof of Proposition \ref{pi}}

From Proposition \ref{operator}, we know that the representation $\pi$ is
unitarily equivalent to
\begin{equation}
\int_{\frac{\mathbb{R}^{d}}{\mathbb{Z}^{d}}\times\frac{\mathbb{R}^{d}%
}{A^{\star}\mathbb{Z}^{d}}}^{\oplus}\mathrm{Ind}_{\Gamma_{0}}^{\Gamma}%
\chi_{\left(  1,\sigma\right)  }\text{ }d\sigma.
\end{equation}
We recall that $\Gamma_{0}$ is isomorphic to the discrete group $\mathbb{Z}%
_{m}\times B\mathbb{Z}^{d}\times A\mathbb{Z}^{d}$ and that $\Gamma_{1}$ is
isomorphic to $\mathbb{Z}_{m}\times B\mathbb{Z}^{d}$ where $m$ is the number
of elements in the commutator group of $\Gamma$ which is a discrete subgroup
of the torus. From Proposition \ref{L}, we have
\begin{equation}
L\simeq{\displaystyle\oplus}_{k=0}^{m-1}\int_{\frac{\mathbb{R}^{d}}{B^{\star
}\mathbb{Z}^{d}}\times\frac{\mathbb{R}^{d}}{A^{\star}\mathbb{Z}^{d}}}^{\oplus
}\mathrm{Ind}_{\Gamma_{0}}^{\Gamma}\left(  \chi_{\left(  k,\sigma\right)
}\right)  \text{ }d\sigma. \label{decL}%
\end{equation}
Now, put
\begin{equation}
L_{k}=\int_{\frac{\mathbb{R}^{d}}{B^{\star}\mathbb{Z}^{d}}\times
\frac{\mathbb{R}^{d}}{A^{\star}\mathbb{Z}^{d}}}^{\oplus}\mathrm{Ind}%
_{\Gamma_{0}}^{\Gamma}\chi_{\left(  k,\sigma\right)  }\text{ }d\sigma.
\end{equation}
From (\ref{decL}), it is clear that $L=L_{0}\oplus\cdots\oplus L_{m-1}.$ Next,
for distinct $i$ and $j,$ the representations $L_{i}$ and $L_{j}$ described
above are disjoint representations. This is due to the fact that if $i\neq j$
then the $\Gamma$-orbits of $\chi_{\left(  i,\sigma\right)  }$ and
$\chi_{\left(  j,\sigma\right)  }$ are disjoint sets and therefore the induced
representations $\mathrm{Ind}_{\Gamma_{0}}^{\Gamma}\chi_{\left(
i,\sigma\right)  }$ and $\mathrm{Ind}_{\Gamma_{0}}^{\Gamma}\chi_{\left(
j,\sigma\right)  }$ are disjoint representations. Thus, for $k\neq1$ the representation $L_k$ must be disjoint from $\pi.$ Let us assume for now that $\left\vert \det B\right\vert >1$ (or $\left\vert \det\left(  B^{\star
}\right)  \right\vert <1$.) According to Lemma \ref{hard}, there exist
measurable cross-sections $\Sigma_{1},\Sigma_{2}$ for $\frac{\mathbb{R}^{d}%
}{\mathbb{Z}^{d}}\times\frac{\mathbb{R}^{d}}{A^{\star}\mathbb{Z}^{d}}$ and
$\frac{\mathbb{R}^{d}}{B^{\star}\mathbb{Z}^{d}}\times\frac{\mathbb{R}^{d}%
}{A^{\star}\mathbb{Z}^{d}}$ respectively such that $\Sigma_{1},\Sigma
_{2}\subset\mathbb{R}^{2},$ $\Sigma_{1}\supset\Sigma_{2}$ and $\Sigma
_{1}-\Sigma_{2}$ is a set of positive Lebesgue measure. Therefore,
\begin{equation}
\pi\simeq\int_{\Sigma_{1}}^{\oplus}\left(  \mathrm{Ind}_{\Gamma_{0}}^{\Gamma
}\left(  \chi_{\left(  1,\sigma\right)  }\right)  \right)  \text{ }%
d\sigma\text{ and }L_{1}\simeq\int_{\Sigma_{2}}^{\oplus}\left(  \mathrm{Ind}%
_{\Gamma_{0}}^{\Gamma}\left(  \chi_{\left(  1,\sigma\right)  }\right)
\right)  \text{ }d\sigma\label{irreducible}%
\end{equation}
and the representations above are realized as acting in the direct integrals
of finite dimensional vector spaces: $\int_{\Sigma_{1}}^{\oplus}l^{2}\left(
\Gamma/\Gamma_{0}\right)  \text{ }d\sigma$ and $\int_{\Sigma_{2}}^{\oplus
}l^{2}\left(  \Gamma/\Gamma_{0}\right)  \text{ }d\sigma$ respectively. We
remark that the direct integrals described in (\ref{irreducible}) are
irreducible decompositions of $\pi$ and $L_{1}$. Now, referring to the central
decomposition of the left regular representation which is described in
(\ref{Left}) there exists a measurable subset $\mathbf{E}$ of $\Sigma_{2}$
such that the central decomposition of $L_{1}$ is given by (see Theorem
$3.26,$ \cite{hartmut})
\[
\int_{\mathbf{E}}^{\oplus}\oplus_{k=1}^{\dim\left(  l^{2}\left(  \Gamma
/\Gamma_{0}\right)  \right)  }\mathrm{Ind}_{\Gamma_{0}}^{\Gamma}\left(
\chi_{\left(  1,\sigma\right)  }\right)  \text{ }d\sigma.
\]
Furthermore, recalling that $L_{1}\simeq\int_{\Sigma_{2}}^{\oplus}\left(
\mathrm{Ind}_{\Gamma_{0}}^{\Gamma}\left(  \chi_{\left(  1,\sigma\right)
}\right)  \right)  \text{ }d\sigma$ and letting $\mu$ be the Lebesgue measure
on $\mathbb{R}^{d}\times\mathbb{R}^{d},$ it is necessarily the case that
\[
\mu\left(  \mathbf{E}\right)  =\frac{\mu\left(  \Sigma_{2}\right)  }%
{\dim\left(  l^{2}\left(  \Gamma/\Gamma_{0}\right)  \right)  }=\frac
{1}{\left\vert \det\left(  B\right)  \det\left(  A\right)  \right\vert
\dim\left(  l^{2}\left(  \Gamma/\Gamma_{0}\right)  \right)  }.
\]
From the discussion provided at the beginning of the third section, the set
$\mathbf{E\ }$is obtained by taking a cross-section for the $\Gamma$-orbits in
$\frac{\mathbb{R}^{d}}{B^{\star}\mathbb{Z}^{d}}\times\frac{\mathbb{R}^{d}%
}{A^{\star}\mathbb{Z}^{d}}.$ Moreover, since $\Sigma_{1}\supset\Sigma_{2}$ and
since $\Sigma_{1}-\Sigma_{2}$ is a set of positive Lebesgue measure then
$\pi\simeq\int_{\Sigma_{1}}^{\oplus}\mathrm{Ind}_{\Gamma_{0}}^{\Gamma}\left(
\chi_{\left(  1,\sigma\right)  }\right)  \text{ }d\sigma\simeq L_{1}\oplus
\int_{\Sigma_{1}-\Sigma_{2}}^{\oplus}\mathrm{Ind}_{\Gamma_{0}}^{\Gamma}\left(
\chi_{\left(  1,\sigma\right)  }\right)  \text{ }d\sigma$ and $\int
_{\Sigma_{1}-\Sigma_{2}}^{\oplus}\mathrm{Ind}_{\Gamma_{0}}^{\Gamma}\left(
\chi_{\left(  1,\sigma\right)  }\right)  $ $d\sigma$ is a subrepresentation of
$\pi.$ Thus a central decomposition of $\pi$ is given by $$\int_{\mathbf{E}%
}^{\oplus}\oplus_{k=1}^{u\left(  \sigma\right)  \dim\left(  l^{2}\left(
\Gamma/\Gamma_{0}\right)  \right)  }\mathrm{Ind}_{\Gamma_{0}}^{\Gamma}\left(
\chi_{\left(  1,\sigma\right)  }\right)  \text{ }d\sigma$$ and the function
$u:\mathbf{E}\rightarrow\mathbb{N}$ is greater than one on a subset of
positive measure of $\mathbf{E}.$ Therefore, according to Theorem $3.26,$
\cite{hartmut}, it is not possible for $\pi$ to be equivalent to a
subrepresentation of the left regular representation of $\Gamma$ if
$\left\vert \det B\right\vert >1.$ Now, let us suppose that $\left\vert \det
B\right\vert \leq1.$ Then $\left\vert \det B^{\star}\right\vert \geq1.$
Appealing to Lemma \ref{hard}, there exist measurable sets $\Sigma_{1}$ and
$\Sigma_{2}$ which are measurable fundamental domains for $\frac
{\mathbb{R}^{d}}{\mathbb{Z}^{d}}\times\frac{\mathbb{R}^{d}}{A^{\star
}\mathbb{Z}^{d}},\text{ and }\frac{\mathbb{R}^{d}}{B^{\star}\mathbb{Z}^{d}%
}\times\frac{\mathbb{R}^{d}}{A^{\star}\mathbb{Z}^{d}}$ respectively, such that
$\Sigma_{1},\Sigma_{2}\subset\mathbb{R}^{d}$ and $\Sigma_{1}\subseteq
\Sigma_{2}$. Next,
\begin{align*}
L_{1}  &  \simeq\int_{\Sigma_{2}}^{\oplus}\mathrm{Ind}_{\Gamma_{0}}^{\Gamma
}\left(  \chi_{\left(  1,\sigma\right)  }\right)  \text{ }d\sigma\\
&  \simeq\left(  \int_{\Sigma_{1}}^{\oplus}\mathrm{Ind}_{\Gamma_{0}}^{\Gamma
}\left(  \chi_{\left(  1,\sigma\right)  }\right)  \text{ }d\sigma\right)
\oplus\left(  \int_{\Sigma_{2}-\Sigma_{1}}^{\oplus}\mathrm{Ind}_{\Gamma_{0}%
}^{\Gamma}\left(  \chi_{\left(  1,\sigma\right)  }\right)  \text{ }%
d\sigma\right) \\
&  \simeq\pi\oplus\left(  \int_{\Sigma_{2}-\Sigma_{1}}^{\oplus}\mathrm{Ind}%
_{\Gamma_{0}}^{\Gamma}\left(  \chi_{\left(  1,\sigma\right)  }\right)  \text{
}d\sigma\right)  .
\end{align*}
Finally, $\pi$ is equivalent to a subrepresentation of $L_{1}$ and is
equivalent to a subrepresentation of the left regular representation $L.$

\subsection{Examples}

In this subsection, we shall present a few examples to illustrate the results
obtained in Propositions \ref{L}, \ref{operator} and \ref{pi}.

\begin{enumerate}
\item Let us start with a trivial example. Let $d=1$ and $B=\frac{2}{3}.$ Then
$B^{\star}=\frac{3}{2},A=3,\text{ and }A^{\star}=\frac{1}{3}.$ Next,
$L\simeq\oplus_{k=0}^{2}\int_{\left[  0,\frac{3}{2}\right)  \times\left[
0,\frac{1}{3}\right)  }^{\oplus}\mathrm{Ind}_{\Gamma_{0}}^{\Gamma}%
\chi_{\left(  k,\sigma\right)  }\text{ }d\sigma$ and $\pi\simeq\int_{\left[
0,1\right)  \times\left[  0,\frac{1}{3}\right)  }^{\oplus}\mathrm{Ind}%
_{\Gamma_{0}}^{\Gamma}\chi_{\left(  1,\sigma\right)  }\text{ }d\sigma.$ Now,
the central decomposition of $L_{1}$ is given by $$\int_{\left[  0,\frac{1}%
{2}\right)  \times\left[  0,\frac{1}{3}\right)  }^{\oplus}\oplus_{j=1}%
^{3}\mathrm{Ind}_{\Gamma_{0}}^{\Gamma}\chi_{\left(  1,\sigma\right)  }\text{
}d\sigma$$ and the central decomposition of the rational Gabor representation
$\pi$ is $$\int_{\left[  0,\frac{1}{2}\right)  \times\left[  0,\frac{1}%
{3}\right)  }^{\oplus}\oplus_{j=1}^{2}\mathrm{Ind}_{\Gamma_{0}}^{\Gamma}%
\chi_{\left(  1,\sigma\right)  }\text{ }d\sigma.$$ From these decompositions,
it is obvious that the rational Gabor representation $\pi$ is equivalent to a
subrepresentation of the left regular representation $L.$

\item If we define $B=\left[
\begin{array}
[c]{cc}%
\frac{2}{3} & 0\\
0 & \frac{3}{2}%
\end{array}
\right] ,$ then $B^{\star}=\left[
\begin{array}
[c]{cc}%
\frac{3}{2} & 0\\
0 & \frac{2}{3}%
\end{array}
\right]  \text{, }A=\left[
\begin{array}
[c]{cc}%
3 & 0\\
0 & 2
\end{array}
\right]  \text{ and }A^{\star}=\left[
\begin{array}
[c]{cc}%
\frac{1}{3} & 0\\
0 & \frac{1}{2}%
\end{array}
\right] .$ Next, the left regular representation of $\Gamma$ can be decomposed
into a direct integral of representations as follows: $$L\simeq\oplus_{k=0}%
^{5}\int_{\mathbf{S}}^{\oplus}\mathrm{Ind}_{\Gamma_{0}}^{\Gamma}\chi_{\left(
k,\sigma\right)  }\text{ }d\sigma$$ where $\mathbf{S=S}_{1}\times A^{\star
}\left[  0,1\right)  ^{2}$ and
\[
\mathbf{S}_{1}=\left(  \left[  0,1\right)  \times\left[  0,\frac{2}{3}\right)
\right)  \cup\left(  \left[  1,\frac{3}{2}\right)  \times\left[  \frac{2}%
{3},1\right)  \right)  \cup\left(  \left[  -\frac{1}{2},0\right)
\times\left[  -\frac{1}{3},0\right)  \right)
\]
is a common connected fundamental domain for the lattices $B^{\star}%
\mathbb{Z}^{2}$ and $\mathbb{Z}^{2}.$

\begin{figure}[h!]
  \caption{Illustration of the set $\mathbf{S}_{1}.$}
  \centering
    \includegraphics[width=0.5\textwidth]{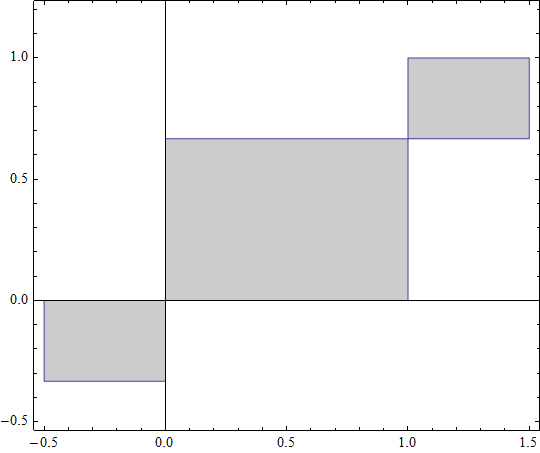}
\end{figure}

\noindent Moreover, we decompose the rational Gabor representation as follows:
$\pi\simeq\int_{\mathbf{S}}^{\oplus}\mathrm{Ind}_{\Gamma_{0}}^{\Gamma}%
\chi_{\left(  1,\sigma\right)  }\text{ }d\sigma.$ One interesting fact to
notice here is that: the rational Gabor representation $\pi$ is actually
equivalent to $L_{1}$ and $$L=L_{0}\oplus L_{1}\oplus L_{2}\oplus L_{3}\oplus
L_{4}\oplus L_{5}.$$

\item Let $\Gamma=\left\langle T_{k},M_{Bl}:k,l\in\mathbb{Z}^{3}\right\rangle
$ where $B=\left[
\begin{array}
[c]{ccc}%
1 & 0 & 0\\
-\frac{1}{5} & \frac{1}{5} & 0\\
1 & -1 & 5
\end{array}
\right] .$ The inverse transpose of the matrix $B$ is given by $B^{\star
}=\left[
\begin{array}
[c]{ccc}%
1 & 1 & 0\\
0 & 5 & 1\\
0 & 0 & \frac{1}{5}%
\end{array}
\right] .$ Next, we may choose the matrix $A$ such that
\[
A=\left[
\begin{array}
[c]{ccc}%
1 & 1 & 0\\
0 & 5 & 5\\
0 & 0 & 1
\end{array}
\right]  \text{ and }A^{\star}=\left[
\begin{array}
[c]{ccc}%
1 & 0 & 0\\
-\frac{1}{5} & \frac{1}{5} & 0\\
1 & -1 & 1
\end{array}
\right]  .
\]
Finally, we observe that $\left[
\begin{array}
[c]{ccc}%
0 & 0 & 1\\
0 & 1 & 5\\
1 & \frac{1}{5} & 0
\end{array}
\right]  \left[  0,1\right)  ^{3}$ is a common fundamental domain for both
$B^{\star}\mathbb{Z}^{3}$ and $\mathbb{Z}^{3}.$ Put
\[
\mathbf{S}=\left[
\begin{array}
[c]{ccc}%
0 & 0 & 1\\
0 & 1 & 5\\
1 & \frac{1}{5} & 0
\end{array}
\right]  \left[  0,1\right)  ^{3}\times\left[
\begin{array}
[c]{ccc}%
1 & 0 & 0\\
-\frac{1}{5} & \frac{1}{5} & 0\\
1 & -1 & 1
\end{array}
\right]  \left[  0,1\right)  ^{3}.
\]
Then $L\simeq\oplus_{k=0}^{4}\int_{\mathbf{S}}^{\oplus}\mathrm{Ind}%
_{\Gamma_{0}}^{\Gamma}\chi_{\left(  k,t\right)  }dt$ and $\pi\simeq
\int_{\mathbf{S}}^{\oplus}\mathrm{Ind}_{\Gamma_{0}}^{\Gamma}\chi_{\left(
1,t\right)  }dt.$
\end{enumerate}

\section{Application to time-frequency analysis}

Let $\pi$ be a unitary representation of a locally compact group $X,$ acting
in some Hilbert space $\mathcal{H}.$ We say that $\pi$ is admissible, if and
only if there exists some vector $\phi\in\mathcal{H}$ such that the operator
$W_{\phi}^{\pi}$ defined by $W_{\phi}^{\pi}:\mathcal{H}\rightarrow
L^{2}\left(  X\right)  ,\text{ }W_{\phi}^{\pi}\psi\left(  x\right)
=\left\langle \psi,\pi\left(  x\right)  \phi\right\rangle $ is an isometry of
$\mathcal{H}$ into $L^{2}\left(  X\right)  .$ 

We continue to assume that $B$ is an invertible rational matrix with at least
one entry which is not an integer. Following Proposition $2.14$ and Theorem
$2.42$ of \cite{hartmut}, the following is immediate. 

\begin{lemma}
\label{admissible} A representation of $\Gamma$ is admissible if and only if
the representation is equivalent to a subrepresentation of the left regular
representation of $\Gamma.$
\end{lemma}

Given a countable sequence $\left\{  f_{i}\right\}  _{i\in I}$ of vectors in a
Hilbert space $\mathbf{H},$ we say $\left\{  f_{i}\right\}  _{i\in I}$ forms a
frame if and only if there exist strictly positive real numbers $A,B$ such
that for any vector $f\in\mathbf{H}$,
\[
A\left\Vert f\right\Vert ^{2}\leq\sum_{i\in I}\left\vert \left\langle
f,f_{i}\right\rangle \right\vert ^{2}\leq B\left\Vert f\right\Vert ^{2}.
\]
In the case where $A=B$, the sequence of vectors $\left\{  f_{i}\right\}
_{i\in I}$ forms a tight frame, and if $A=B=1,$ $\left\{  f_{i}\right\}
_{i\in I}$ is called a Parseval frame. We remark that an admissible vector for
the left regular representation of $\Gamma$ is a Parseval frame by definition.

The following proposition is well-known for the more general case where $B$ is
any invertible matrix (not necessarily a rational matrix.) Although this
result is not new, the proof of Proposition \ref{ratmatrix} is new, and worth
presenting in our opinion.

\begin{proposition}
\label{ratmatrix}Let $B$ be a rational matrix. There exists a vector $g\in
L^{2}\left(  \mathbb{R}^{d}\right)  $ such that the system $\left\{
M_{l}T_{k}g:l\in B\mathbb{Z}^{d},k\in\mathbb{Z}^{d}\right\}  $ is a Parseval
frame in $L^{2}\left(  \mathbb{R}^{d}\right)  $ if and only if $\left\vert
\det B\right\vert \leq1.$
\end{proposition}

\begin{proof}
The case where $B$ is an element of $GL\left(  d,\mathbb{Z}\right)  $ is
easily derived from \cite{moussa}, Section $4$. We shall thus skip this case.
So let us assume that $B$ is a rational matrix with at least one entry not in
$\mathbb{Z}$. We have shown that the representation $\pi$ is equivalent to a
subrepresentation of the left regular representation of $L$ if and only if
$\left\vert \det B\right\vert \leq1.$ Since $\Gamma$ is a discrete group, then
its left regular representation is admissible if and only if $\left\vert \det
B\right\vert \leq1.$ Thus, the representation $\pi$ of $\Gamma$ is admissible
if and only if $\left\vert \det B\right\vert \leq1.$ Suppose that $\left\vert
\det B\right\vert \leq1.$ Then $\pi$ is admissible and there exists a vector
$f\in L^{2}\left(  \mathbb{R}^{d}\right)  $ such that the map $W_{f}^{\pi}$
defined by $W_{f}^{\pi}h\left(  e^{2\pi i\theta}M_{l}T_{k}\right)
=\left\langle h,e^{2\pi i\theta}M_{l}T_{k}f\right\rangle $ is an isometry. As
a result, for any vector $h\in L^{2}\left(  \mathbb{R}^{d}\right)  ,$ we have
\[
\left(  {\displaystyle\sum\limits_{\theta\in\left[  \Gamma,\Gamma\right]  }%
}{\displaystyle\sum\limits_{l\in B\mathbb{Z}^{d}}}\sum_{k\in\mathbb{Z}^{d}%
}\left\vert \left\langle h,e^{2\pi i\theta}M_{l}T_{k}f\right\rangle
\right\vert ^{2}\right)  ^{1/2}=\left\Vert h\right\Vert _{L^{2}\left(
\mathbb{R}^{d}\right)  }.
\]
Next, for $m=\mathrm{card}\left(  \left[  \Gamma,\Gamma\right]  \right)  ,$%
\[
{\displaystyle\sum\limits_{\theta\in\left[  \Gamma,\Gamma\right]  }%
}{\displaystyle\sum\limits_{l\in B\mathbb{Z}^{d}}}\sum_{k\in\mathbb{Z}^{d}%
}\left\vert \left\langle h,e^{2\pi i\theta}M_{l}T_{k}f\right\rangle
\right\vert ^{2}={\displaystyle\sum\limits_{l\in B\mathbb{Z}^{d}}}\sum
_{k\in\mathbb{Z}^{d}}\left\vert \left\langle h,M_{l}T_{k}\left(
m^{1/2}f\right)  \right\rangle \right\vert ^{2}.
\]
Therefore, if $g=m^{1/2}f$ then $$\left(  {\displaystyle\sum\limits_{l\in
B\mathbb{Z}^{d}}}\sum_{k\in\mathbb{Z}^{d}}\left\vert \left\langle h,M_{l}%
T_{k}g\right\rangle \right\vert ^{2}\right)  ^{1/2}=\left\Vert h\right\Vert
_{L^{2}\left(  \mathbb{R}^{d}\right) }.$$ For the converse, if we assume that
there exists a vector $g\in L^{2}\left(  \mathbb{R}^{d}\right)  $ such that
the system $$\left\{  M_{l}T_{k}g:l\in B\mathbb{Z}^{d},k\in\mathbb{Z}%
^{d}\right\} $$ is a Parseval frame in $L^{2}\left(  \mathbb{R}^{d}\right)  $
then it is easy to see that $\pi$ must be admissible. As a result, it must be
the case that $\left\vert \det B\right\vert \leq1.$
\end{proof}

\subsection{Proof of Proposition \ref{main2}}

Let us suppose that $\left\vert \det B\right\vert \leq1.$ From the proof of
Proposition \ref{pi}, we recall that there exists a unitary map $\mathfrak{A:}%
\int_{\mathbf{E}}^{\oplus}\left(  \oplus_{k=1}^{\ell\left(  \sigma\right)
}l^{2}\left(  \frac{\Gamma}{\Gamma_{0}}\right)  \right)  d\sigma\rightarrow
L^{2}\left(  \mathbb{R}^{d}\right) $ which intertwines the representations
$\int_{\mathbf{E}}^{\oplus}\left(  \oplus_{k=1}^{\ell\left(  \sigma\right)
}\mathrm{Ind}_{\Gamma_{0}}^{\Gamma}\left(  \chi_{\left(  1,\sigma\right)
}\right)  \right)  d\sigma\text{ with }\pi$ such that $\int_{\mathbf{E}%
}^{\oplus}\left(  \oplus_{k=1}^{\ell\left(  \sigma\right)  }\mathrm{Ind}%
_{\Gamma_{0}}^{\Gamma}\left(  \chi_{\left(  1,\sigma\right)  }\right)
\right)  d\sigma$ is the central decomposition of $\pi,$ and $\mathbf{E\subset
\mathbb{R}}^{d}$ is a measurable subset of a fundamental domain for the
lattice $B^{\star}\mathbb{Z}^{d}\times A^{\star}\mathbb{Z}^{d}$ and the
multiplicity function $\ell$ satisfies the condition: $\ell\left(
\sigma\right)  \leq\dim l^{2}\left(  \frac{\Gamma}{\Gamma_{0}}\right)  .$
Next, according to the discussion on Page $126,$ \cite{hartmut} the vector $a$
is an admissible vector for the representation $\tau=\int_{\mathbf{E}}%
^{\oplus}\left(  \oplus_{k=1}^{\ell\left(  \sigma\right)  }\mathrm{Ind}%
_{\Gamma_{0}}^{\Gamma}\left(  \chi_{\left(  1,\sigma\right)  }\right)
\right)  d\sigma$ if and only if $a\in\int_{\mathbf{E}}^{\oplus}\left(
\oplus_{k=1}^{\ell\left(  \sigma\right)  }l^{2}\left(  \frac{\Gamma}%
{\Gamma_{0}}\right)  \right)  d\sigma$ such that for $d\sigma$-almost every
$\sigma\in\mathbf{E},$ $\left\Vert a\left(  \sigma\right)  \left(  k\right)
\right\Vert _{l^{2}\left(  \frac{\Gamma}{\Gamma_{0}}\right)  }^{2}=1\text{ for
}1\leq k\leq\ell\left(  \sigma\right) $ and for distinct $k,j\in\left\{
1,\cdots,\ell\left(  \sigma\right)  \right\}  $ we have $\left\langle a\left(
\sigma\right)  \left(  k\right)  ,a\left(  \sigma\right)  \left(  j\right)
\right\rangle =0.$ Finally, the desired result is obtained by using the fact
that $\mathfrak{A}$ intertwines the representations $\tau$ with $\pi.$


\end{document}